\documentclass[12pt,amscd]{amsart}
\usepackage{tikz}
\usepackage[all]{xy}
\usepackage{graphicx}
\usepackage{amsmath,amsxtra,amssymb,latexsym, amscd,amsthm}
\usepackage{indentfirst}
\usepackage[mathscr]{eucal}
\usepackage[pagebackref=true]{hyperref}

\newtheorem{thm}{Theorem}[section]
\newtheorem{cor}[thm]{Corollary}
\newtheorem{lem}[thm]{Lemma}

\theoremstyle{definition}
\newtheorem{defn}[thm]{Definition}
\newtheorem{exm}[thm]{Example}

\newtheorem{rem}[thm]{Remark}

\newtheorem{conj}[thm]{Conjecture}

\numberwithin{equation}{section}

\DeclareMathOperator{\NN}{\mathbb {N}}
\DeclareMathOperator{\ZZ}{\mathbb {Z}}
\DeclareMathOperator{\RR}{\mathbb {R}}

\DeclareMathOperator{\lk}{lk}

\DeclareMathOperator{\supp}{supp}
\DeclareMathOperator{\NP}{NP}

\DeclareMathOperator{\reg}{reg}

\def\D {\Delta}

\def\T {\mathcal T}
\def\U {\mathcal U}

\def\W {\mathcal W}

\def\a {\mathbf a}
\def\b {\mathbf b}
\def\c {\mathbf c}

\def\e {\mathbf e}

\def\w {\mathbf w}

\def\m {\mathfrak m}
\def\k {\mathrm k}

\def\h {\widetilde{H}}

\begin{document}

\title[Integral closure weighted oriented graphs]{Regularity of integral closures of edge ideals of weighted oriented graphs}

\textbf{\author{Nguyen Cong Minh}
	\address{Faculty of Mathematics and Informatics, Hanoi University of Science and Technology, 1 Dai Co Viet, Hanoi, Vietnam}
	\email{minh.nguyencong@hust.edu.vn}}

\author{Thanh Vu}
\address{Institute of Mathematics, VAST, 18 Hoang Quoc Viet, Hanoi, Vietnam}
\email{vuqthanh@gmail.com}

\author{Guangjun Zhu}
\address{School of Mathematical Sciences, Soochow University, Suzhou, Jiangsu, 215006, P. R. China}
\email{zhuguangjun@suda.edu.cn}


\date{}

\dedicatory{Dedicated to Professor Bernd Ulrich on the occasion of his 70th birthday}

\begin{abstract}
    We prove that the regularity cannot increase when taking the integral closure for edge ideals of arbitrary weighted oriented graphs.
\end{abstract}

\maketitle

\section{Introduction}
\label{sect_intro}

K\"uronya and Pintye \cite{KP} studied the relationship between the log-canonical thresholds and the Castelnouvo-Mumford regularity of coherent sheaves of ideals in projective spaces and made the following striking conjecture
\begin{conj}\label{conj_KP} Let $I$ be a homogeneous ideal in a polynomial ring $S = \k[x_1,\ldots,x_n]$ over a field $\k$. Denote by $\overline{I}$ the integral closure of $I$. Then 
$$\reg (\overline{I}) \le \reg (I),$$
where  $\reg$ denotes the Castelnuovo-Mumford regularity.
\end{conj}
Integral closures of ideals also play an important role in the multiplicity and singularity theory \cite{PTUV, R, T, UV}. To our knowledge, very little is known about this conjecture even for monomial ideals, partly due to the fact that computing the integral closure of an ideal is a difficult problem. In \cite{MV1}, Minh and Vu proved that $\reg (\overline{I^s}) = \reg (I^s)$ for natural exponents $s \le 4$ for arbitrary edge ideals of simple graphs and arbitrary natural exponents $s$ when $I$ is the Stanley-Reisner ideal of a one-dimensional simplicial complex. In this paper, we prove Conjecture \ref{conj_KP} for edge ideals of arbitrary weighted oriented graphs. Let us recall the notion of weighted oriented graphs and their edge ideals. 

Let $(D, E,\w)$ be a weighted oriented graph without isolated vertices, where $\w: V(D) \to \ZZ_+$ is a weight function on the vertex set of $D$. Assume that $D$ has $|V(D)| = n$ vertices. We then identify the set of vertices of $D$ with the set $[n] = \{1,\ldots,n\}$. The edge ideal of $(D,E,\w)$ is defined by  
$$I(D,\w) = \left ( x_i x_j^{\w(j)} \mid (i,j) \in E(D) \right ) \subseteq S = \k[x_1,\ldots,x_n].$$ 
The edge ideal of a weighted oriented graph was introduced in \cite{GMSVV, PRT}, where the authors described the primary decomposition of $I(D,\w)$ and studied the unmixed and Cohen-Macaulay properties of $I(D,\w)$. When $\w(j) = 1$ for all $j \in [n]$, then $I(D,\w) = I(G)$ is the usual edge ideal of the underlying undirected graph $G$ associated with $D$. Hence, we assume that $\w$ is nontrivial, i.e., $\w(j) > 1$ for some vertex $j \in [n]$, which is not a source vertex of $D$.

In \cite{MVZ}, we classified all integrally closed edge ideals of weighted oriented graphs. Indeed, when the weights are nontrivial, most of them are not integrally closed. In this paper, we prove
\begin{thm}\label{main_thm}
    Let $(D, E,\w)$ be an arbitrary weighted oriented graph. Then 
    $$\reg \left ( \overline{I(D ,\w)} \right ) \le \reg (I(D,\w)).$$
\end{thm}

Our proof relies on a careful analysis of the degree complexes associated with the integral closure of $I(D,\w)$ and a framework for comparing the regularity of monomial ideals laid out in \cite{MNPTV}. We further compute the regularity of integral closures of arbitrary weighted oriented complete graphs.

\begin{thm}\label{thm_complete} Let $(D,E,\w)$ be a weighted oriented complete graph on $n \ge 2$ vertices. Then 
$$\reg \left ( \overline{ I(D,\w)} \right ) = \max \{ \w(j) \mid j = 1, \ldots, n \} + 1.$$    
\end{thm}

\section{Degree complexes and integral closures of monomial ideals}

In this section, we discuss the Stanley-Reisner correspondence, degree complexes of monomial ideals, Castelnuovo-Mumford regularity, and integral closures of monomial ideals. We refer to \cite{MNPTV, MV1} for more information.

Throughout the paper, we denote by $S = \k[x_1,\ldots,x_n]$ a standard graded polynomial ring over an arbitrary field $\k$. 

\subsection{Stanley-Reisner correspondence}
Let $\Delta$ be a simplicial complex on $[n] = \{1,\ldots,n\}$. For a face $F \in \Delta$, the link of $F$ in $\Delta$ is defined by 
$$ \lk_{\Delta} F = \{G \in \Delta \mid F \cup G \in \Delta, F \cap G = \emptyset \}.$$
For each subset $F$ of $[n]$, let $x_F = \prod_{i\in F} x_i$ be a squarefree monomial of $S$. We now recall the Stanley-Reisner correspondence. 

\begin{enumerate}
    \item For a squarefree monomial ideal $I \subseteq S$, the Stanley-Reinser complex of $I$ is defined by $\Delta(I) = \{ F \subseteq [n] \mid x_F \notin I\}$. 
    \item For a simplicial complex $\Delta$ on the vertex set $[n]$, the Stanley-Reisner ideal of $\Delta$ is defined by $I_\Delta = (x_F \mid F \notin \Delta)$. 
    \item The Stanley-Reisner ring of $\Delta$ is $\k[\Delta] = S/ I_\Delta$.
\end{enumerate}
From the definition, we have the following simple properties, see e.g., \cite{S}.
\begin{lem}\label{lem_SR} Let $I, J$ be squarefree monomial ideals of $S = \k[x_1, \ldots,x_n]$. Then
    \begin{enumerate}
        \item $\Delta(I)$ is a cone over $t \in [n]$ if and only if $x_t$ does not divide any minimal generator of $I$.
        \item $\Delta (I + J) = \Delta(I) \cap \Delta(J)$.
        \item $\Delta(I \cap J) = \Delta(I) \cup \Delta(J)$.
    \end{enumerate}
\end{lem}
\begin{defn}
    The $q$-th reduced homology group of $\Delta$ with coefficients $\k$, denoted by $\h_q(\Delta;\k)$, is the $q$-th homology group of the augmented oriented chain complex of $\Delta$ over $\k$.
\end{defn}

A simplicial complex $\D$ is called {\it acyclic} if $\h_i(\Delta;\k) = 0$ for all $i$.
\begin{rem} Let $\Delta$ be a simplicial complex. Then \begin{enumerate}
    \item $\h_{-1}(\Delta;\k) \neq 0$ if and only if $\Delta$ is the empty complex (i.e., $\Delta=\{\emptyset\}$).
    \item If $\Delta$ is a cone over some $t \in [n]$ or $\Delta$ is the void complex  (i.e., $\Delta=\emptyset$), then it is acyclic.
\end{enumerate}
\end{rem}

\subsection{Castelnuovo-Mumford regularity  and degree complexes of monomial ideals} Let $\m = (x_1,\ldots,x_n)$ be the maximal homogeneous ideal of $S$. For a finitely generated graded $S$-module $L$, the Castelnuovo-Mumford regularity (or regularity for short) of $L$ is 
$$\reg (L) = \max \{i + j \mid H_{\m}^i(L)_j \neq 0\},$$
where $H_\m^i(L)$ denotes the $i$-th local cohomology module of $L$ with respect to $\m$. For a non-zero proper homogeneous ideal $J$ of $S$, we have $\reg (J) = \reg (S/J) + 1$. When $I$ is a monomial ideal, we have 
\begin{lem}\label{lem_reg_add_vars} Let $I$ be a monomial ideal and $x$ be a variable of $S$. Then 
$$\reg (I + (x)) \le \reg (I).$$    
\end{lem}
\begin{proof}
    See \cite[Corollary 4.8]{CHHKTT} or \cite[Corollary 2.20]{MNPTV}.
\end{proof}

To proceed further, we introduce some notation. For an exponent $\a = (a_1, \ldots,a_n) \in \NN^n$, we denote by $x^\a$ the monomial $x_1^{a_1} \cdots x_n^{a_n}$ in $S$ and $|\a| = a_1 + \cdots + a_n$. The support of $\a$ is $\supp \a = \{ i \in [n] \mid a_i \neq 0\}.$

\begin{defn} Let $I$ be a monomial ideal of $S$ and $\a = (a_1, \ldots, a_n) \in \NN^n$ be an exponent. The degree complex of $I$ in degree $\a$ is defined by 
$$\Delta_{\a}(I) = \Delta ( \sqrt{I : x^\a}).$$    
\end{defn}

An ideal of the form $\sqrt{I:x^\a}$ is called an associated radical of $I$. Associated radicals of monomial ideals play a fundamental role in studying the depth of monomial ideals \cite{Hoc}. They also appear in the following result of Minh, Nam, Phong, Thuy, and Vu \cite{MNPTV} for computing the regularity of monomial ideals.

\begin{lem}\label{Key0}
Let $I$ be a monomial ideal in $S$. Then
\begin{multline*}
\reg(S/I)=\max\{|\a|+i \mid \a\in\NN^n,i\ge 0,\h_{i-1}(\lk_{\D_\a(I)}F;\k)\ne 0\\ \text{ for some $F\in \D_\a(I)$ with $F\cap \supp \a=\emptyset$}\}.
\end{multline*}
\end{lem}
\begin{proof}
    Follows from Lemma 2.12 and Lemma 2.19 of \cite{MNPTV}.
\end{proof}

\begin{defn} A pair $(\a,i) \in \NN^n \times \NN$ is called a {\it critical pair} of $I$ if there exists a face $F$ of $\D_\a(I)$ with $F \cap \supp \a = \emptyset$ and an index $i$ such that $\h_{i-1}(\lk_{\D_\a(I)} F;\k) \neq 0$. It is called an {\it extremal pair} of $I$ if furthermore, $\reg (S/I) = |\a| + i$.    

An exponent $\a \in \NN^n$ is called a {\it critical exponent} (respectively an {\it extremal exponent}) of $I$ if $(\a,i)$ is a critical pair (respectively an extremal pair) of $I$ for some $i \ge 0$.
\end{defn}

For a monomial $f$ in $S$ and $j \in [n]$, $\deg_j(f)$ denotes the degree of $x_j$ in $f$. For a monomial ideal $I$, $\rho_j(I)$ is defined by
 $$\rho_j(I) = \max \{\deg_j(u) \mid u \text{ is a minimal monomial generator of } I\}.$$
 By \cite[Remark 2.13]{MNPTV}, we have 
 \begin{rem}\label{rem_critical} A critical exponent $\a$ of $I$ belongs to the finite set 
 $$\Gamma(I) = \{ \a \in \NN^n \mid a_j < \rho_j(I) \text{ for all } j = 1, \ldots, n\}.$$
 \end{rem}
 
The following lemma is \cite[Lemma 2.9]{HV}; we provide an argument here for completeness.

\begin{lem}\label{lem_regularity_equality_adding_variable} Let $I$ be a monomial ideal. Assume that $(\a,i)$ is an extremal pair of $I$ and $F$ is a face of $\Delta_\a(I)$ such that $F \cap \supp \a = \emptyset$ and $\h_{i-1}(\lk_{\Delta_\a(I)} F;\k) \neq 0$. Assume that $j \in F$. Then $\reg (S/I) = \reg (S/(I,x_j))$.    
\end{lem}
\begin{proof}
    Let $J = I + (x_j)$. Since $j\in F$, $j \notin \supp \a$. By \cite[Lemma 2.24]{MNPTV}, $\sqrt{J:x^\a} = \sqrt{I:x^\a} + (x_j)$. In other words, $\Delta_\a(J)$ is the restriction of $\Delta_\a(I)$ to $[n]\setminus \{j\}$. Let $F' = F \setminus \{j\}$. For simplicity of notation, we set $\Delta = \lk_{\Delta_\a(I)} F'$ and $\Gamma = \lk_{\Delta_\a(J)} F'$. Then we have $\lk_{\Delta_\a(I)} F = \lk_{\Delta} \{j\}$ and $\Delta = \Gamma \cup \{j\} * \lk_{\Delta} \{j\}.$ Note that $\h_i(\Delta) =0$, otherwise, $(\a,i+1)$ is a critical pair of $I$, which is a contradiction to Lemma \ref{Key0}. From the Mayer-Vietoris sequence, we deduce a long exact sequence
    $$0 \to \h_{i-1} (\lk_\Delta \{j\} ;\k) \to \h_{i-1} (\Gamma;\k) \to \cdots.$$
    Hence, $\h_{i-1}(\Gamma;\k) \neq 0$. In other words, $(\a,i)$ is a critical pair of $J$. By Lemma \ref{lem_reg_add_vars} and Lemma \ref{Key0}, we have
    $$\reg (S/J) \ge |\a| + i = \reg (S/I) \ge \reg (S/J).$$ 
    The conclusion follows.
\end{proof}

\subsection{Integral closures of monomial ideals}
We recall the definition and some properties of integral closures of monomial ideals; see \cite[Section 4]{E} and \cite{SH} for more details. We refer to \cite{Hoa} and \cite{Tr} for beautiful expositions on the regularity and depth of integral closures of powers of monomial ideals.

\begin{defn} Let $I$ be a monomial ideal of $S$. The exponent set of $I$ is $E(I) = \{ \a \in \NN^n \mid x^\a \in I\}$. The Newton polyhedron of $I$, denoted by $\NP(I)$, is the convex hull of the exponent set of $I$ in $\RR^n$.
\end{defn}
The following result is standard, see e.g. \cite[Exercise 4.23]{E}.
\begin{lem}\label{lem_newton_polyhedra} Let $I$ be a monomial ideal of $S$. The integral closure of $I$ is a monomial ideal with exponent set $E(\overline{I}) = \NP(I) \cap \ZZ^n$.    
\end{lem} 
For $\a = (a_1, \ldots, a_n),\b = (b_1, \ldots, b_n)\in \RR^n$, the notation $\a \ge \b$ means that $a_j \ge b_j$ for all $j = 1, \ldots, n$.
\begin{lem}\label{lem_criterion_integral} Let $I$ be a monomial ideal and $\a \in \NN^n$ be an exponent. Then $x^\a \in \overline{I}$ if and only if there exist non-negative rational numbers $c_j$ and exponents $\b_j$ of $E(I)$ for $j = 1, \ldots, s$ such that $\sum_{j=1}^s c_j \ge 1$ and $\a \ge \sum_{j = 1}^s c_j \b_j$.    
\end{lem}
\begin{proof}
    The conclusion follows from Lemma \ref{lem_newton_polyhedra}.
\end{proof}
For a monomial $f$ in $S$, the support of $f$, denoted by $\supp (f)$, is the set of all indices $i \in [n]$ such that $x_i \mid f$. For a monomial ideal $J$ of $S$ and a subset $V$ of $[n]$, the restriction of $J$ to $V$, denoted by $J_V$, is 
$$J_V = (f \mid f \text{ is a minimal generator of } J \text{ such that } \supp f \subseteq V).$$

We have 
\begin{cor}\label{cor_restriction}
    Let $I$ be a monomial ideal. Then 
    $$ \overline{ I_V } = (\overline{I})_V.$$ 
\end{cor}

\section{Edge ideals of weighted oriented graphs and their integral closures}
In this section, we prove Theorem \ref{main_thm} and compute the regularity of integral closures of edge ideals of weighted oriented complete graphs.

\subsection{Edge ideals of weighted oriented graphs and their associated radicals} 

Let $D$ denote a simple oriented graph without isolated vertices over the vertex set $V(D) = [n] = \{1,\ldots,n\}$ and the edge set $E(D)$. We denote by $G$ the underlying undirected graph associated with $D$. In other words, $V(G) = V(D) = [n]$ and $\{i,j\} \in E(G)$ if and only if either $(i,j)$ or $(j,i)$ is an edge of $D$. A subset $U \subseteq V(D)$ is called an independent set of $D$ if it is an independent set of $G$. For a vertex $i \in V(D)$, the in and out neighborhoods of $i$ are defined by 
$$N_D^-(i) = \{j \in V(D) \mid (j,i) \in E(D)\},~ N_D^+(i) = \{j \in V(D) \mid ( i,j) \in E(D)\}.$$
When $N_D^-(i)  = \emptyset$ (respectively $N_D^+(i) = \emptyset$), we call $i$ a source vertex (respectively a sink vertex) of $D$. For a subset $U$ of $V(D)$, we denote by $N_D^-(U)$ and $N_D^+(U)$ the in and out neighborhoods of $U$. We refer to \cite{D} for more information on graph theory.

Let $\w: V(D) \to \ZZ_+$ be a weight function on the vertices of $D$ such that $\w(j) = 1$ when $j$ is a source vertex. Let $V^+(D)$ be the set of vertices of $D$ with nontrivial weights, i.e., $\w(j) > 1$. The edge ideal of $(D,E,\w)$ is defined by
$$I(D,\w)=(x_ix_j^{\w(j)} \mid (i,j) \in E(D))\subseteq S.$$
When $\w(j) = 1$ for all $j$, we have $I(D,\w) = I(G)$ is the usual edge ideal of $G$. We also denote by $I(D)$ the edge ideal of $D$, i.e., $I(D) = I(G)$. By definition, we have $\rho_j(I(D,\w)) = \w(j)$ for all $j \in V(D)$. By Remark \ref{rem_critical} and \cite[Lemma 2.9]{MV1}, the critical exponents $\a$ of $I(D,\w)$ and $\overline{I(D,\w)}$ satisfy $a_j < \w(j)$ for all $j \in V(D)$. We start by analyzing associated radicals of edge ideals of weighted oriented graphs.

\begin{lem}\label{lem_associated_radicals_weight} Let $(D,E,\w)$ be a weighted oriented graph. Let $\a \in \NN^n$ be an exponent such that $a_j < \w(j)$ for all $j \in V(D)$. We denote by $U = N_D^+(\supp \a)$ and $G\backslash U$ the induced subgraph of $G$ on $V(G) \setminus U$. Then
$$\sqrt{I(D, \w) : x^\a} = I(G \backslash U) + (x_i \mid i \in U).$$
\end{lem}
\begin{proof} Let $J = \sqrt{I(D,{\w}) : x^\a}$. By \cite[Lemma 2.24]{MNPTV}, the generators of $J$ are $x_kx_l$ with $\{k,l\} \in E(G)$ and $x_i$ for some $i \in [n]$. Since $a_i < \w(i)$, we deduce that $x_i \in J$ if and only if there exists an index $j$ such that 
$$x_i = \sqrt{x_jx_i^{\w(i )}/\gcd(x_jx_i^{\w(i)} ,x^\a) }.$$
In particular, we must have $a_j > 0$. The conclusion follows.
\end{proof}

We now define induced weight functions on induced oriented subgraphs of $D$. Let $\w: V(D) \to \ZZ_+$ be a weight function on the vertices of $D$ and $D_U$ be the induced oriented subgraph of $D$ on $U \subseteq V(D)$. The induced weight function of $\w$ on $V(D_U)$, denoted by $\w_U$, is defined by
$$\w_U(j) = \begin{cases} 
\w(j) & \text{ if } j \text{ is not a source vertex in } D_U,\\
1 & \text{ if } j \text{ is a source vertex in } D_U.
\end{cases}$$

\begin{lem}\label{lem_restriction_weighted} Let $U$ be a subset of $V(D)$. Let $D_U$ and $\w_U$ be the induced subgraph of $D$ and the induced weight function of $\w$ on $U$. Then 
$$I(D,\w) + (x_i \mid i \notin U) = I(D_U,\w_U) +  (x_i \mid i \notin U).$$ 
\end{lem}
\begin{proof} It suffices to note that for any $j,k \in U$, $x_jx_k^{\w(k)}$ is a generator of $I(D,\w)$ if and only if $x_j x_k^{\w(k)}$ is a generator of $I(D_U,\w_U)$.     
\end{proof}

\subsection{Integral closures of edge ideals of weighted oriented graphs}
In this subsection, we first prove a technical result about associated radicals of integral closures of edge ideals of weighted oriented graphs. We then prove Theorem \ref{main_thm}. Let $(D,E,\w)$ be a weighted oriented graph and $\a$ be an exponent such that $U = \supp \a$ is an independent set of $D$ consisting of sink vertices only. For each subset $W$ of $U$, we define the $\a$-capacity of $W$ by 
$$c(W) = \sum_{j \in W} \frac{a_j}{\w(j)}.$$
For any subset $Z \subseteq V(D)$, we denote by $(Z)$ the ideal $(Z) = ( x_j \mid j \in Z)$. Then, we define the ideal $n(W)$ by 
$$n(W) = \bigcap_{j \in W} (N_D^-(j)) = \bigcap_{j \in W} ( x_k \mid k \in N_D^-(j)).$$
Note that $j\in W$ is a sink vertex, so $N_D(j) = N_D^-(j)$.
\begin{lem}\label{lem_associated_radicals_integral} Assume that $U = \supp \a$ is an independent set consisting of sink vertices of $D$ only. We denote by $\U$ the set of minimal subsets $W \subseteq U$ such that $c(W) \ge 1$. Then 
$$\sqrt{\overline{I(D,\w)} : x^\a} = I(G) + \sum_{W \in \U} n(W).$$
\end{lem}
Let us consider an example before proving this technical result. 
\begin{exm}
    Let $(D,E,\w)$ be a weighted oriented graph whose edge ideal is
    $$I = (x_1x_2,x_1x_3,x_1x_4,x_2x_5,x_2x_6,x_2x_7^4,x_2x_{10}^6,x_3x_7^4,x_4x_7^4,x_4x_8^7,x_5x_8^7,x_5x_9^4,x_6x_{10}^6).$$ 
    The nontrivial weights of $D$ are $\w(7) = 4, \w(8) = 7, \w(9) = 4$ and $\w(10) = 6$. Let $x^\a = x_7^2x_8^3x_9 x_{10}^3$. In particular, $U = \{7,8,9,10\}$ is an independent set consisting of sink vertices of $D$ only. Then, we can check that the minimal subsets $W \subseteq U$ for which $c(W) \ge 1$ are $\{7,8,9\}$, $\{7,10\}$ and $\{8,9,10\}$. Furthermore, we have 
    $$N_D^-(7) = \{2,3,4\}, N_D^-(8) = \{4,5\}, N_D^-(9) = \{5\}, \text{ and } N_D^{-}(10) = \{2,6\}.$$
    Using Macaulay2 \cite{M2}, we can check that 
    \begin{align*}
        \sqrt{\overline{I}:x^\a} & = \sqrt{I} + (x_2, x_3x_5,x_3x_6,x_4x_5,x_4x_6,x_5x_6) \\
        &= \sqrt{I} + n(\{7,8,9\}) + n(\{7,10\}) + n(\{8,9,10\}).
    \end{align*}
\end{exm}

\begin{proof}[Proof of Lemma \ref{lem_associated_radicals_integral}] We denote by $\e_1, \ldots, \e_n$ the canonical basis of $\RR^n$. First, we prove that $n(W) \subseteq \sqrt{\overline{I(D,\w)} : x^\a}$ for all subset $W \subseteq U$ such that $c(W) \ge 1$. Let $f = x^\b$ be a minimal generator of $n(W)$. We may assume that $\supp \b$ is an independent set. By Corollary \ref{cor_restriction}, we may assume that $V(D) = W \cup \supp \b$. Since $W$ is an independent set, the underlying graph $G$ is bipartite. For simplicity of notation, we assume that $\supp \a = W = \{1, \ldots, s\}$ and $\supp \b = \{s+1, \ldots, n\}$. By assumption, we have 
$$c(W) = \sum_{j = 1}^s \frac{a_{j}}{\w(j)} \ge 1.$$
In particular, we may choose positive rational numbers $\gamma_j \le a_j / \w(j)$ for all $j = 1, \ldots, s$ such that $\sum_{j=1}^s \gamma_j = 1$. For each edge $(k,j)$ in $D$ with $k \in \supp \b$ and $j \in \supp \a$, we set $\gamma_{kj} = \gamma_j / |N^-_D(j)|$. With this definition, we have 
$$\sum_{k \in \supp \b, j \in N^+_D(k)} \gamma_{kj} = 1.$$ 
Furthermore, each edge $(k,j)$ gives rise to a vector $\c_{kj} = \e_k + \w(j) \e_j \in \NP(I(D,\w))$. Let $\c = \sum_{kj} \gamma_{kj} \c_{kj}$. For all $\ell = 1, \ldots, s$, we have 
$$a_\ell \ge \left ( \sum_{k \in N^-_D(\ell)} \gamma_{k\ell} \right )  \w(\ell) = c_\ell.$$
By Lemma \ref{lem_criterion_integral}, we deduce that $x_1^{a_1} \cdots x_s^{a_s} x_{s+1} \cdots x_n \in \overline{I(D,\w)}$. 

Conversely, let $f = x^\b$ be a minimal generator of $\sqrt{\overline{I(D,\w)} : x^\a}$. We need to prove that $f \in I(G) + \sum_{W \in \U} n(W)$. Again, we may assume that $\supp \b$ is an independent set. By Lemma \ref{lem_criterion_integral}, there exist positive rational numbers $\beta_1, \ldots, \beta_t$ and vectors $\c_1, \ldots,\c_t$ corresponding to minimal generators of $I(D,\w)$ such that 
\begin{equation}\label{eq_3}
    \sum_{i=1}^t \beta_i = 1 \text{ and } a_j \ge \sum_{i=1}^t \beta_i c_{ij} \text{ for all } j \in [n] \backslash \supp \b,
\end{equation}
where $a_j = 0$ if $j \notin \supp \a \cup \supp \b$. In particular, $\supp \c_i \subseteq \supp \a  \cup \supp \b$. Since $\supp \a$ and $\supp \b$ are independent sets, we deduce that $\c_i = \e_k + \w(j) \e_j$ for some $k \in \supp \b$ and $j \in \supp \a$. Let $W = \cup_{i=1}^t \supp \c_i \setminus \supp \b$. Then, we have $W \subseteq \supp \a$. We may replace $\c_i$ by $\c_{kj}$ and $\beta_i$ by $\beta_{kj}$ for $k \in \supp \b\backslash \supp \a$ and $j \in N_D^+(k)$. Eq. \eqref{eq_3} can be rewritten as 
$$\sum_{kj} \beta_{kj} = 1 \text{ and } a_j \ge \sum_{kj} \beta_{kj} \w(j).$$
Equivalently, 
$$c(W) = \sum_{j \in W} \frac{a_j}{\w(j)} \ge 1.$$
The conclusion follows.
\end{proof}
We also have a simple lemma. 
\begin{lem}\label{lem_intersection} Let $I, J$, and $K$ be monomial ideals of $S$. Then 
$$ I + (J\cap K) = (I + J) \cap (I + K).$$    
\end{lem}
\begin{proof}
    The conclusion follows from the fact that for monomial ideals, intersection commutes with addition.
\end{proof}
\begin{lem}\label{lem_trivial_homology} Let $I = I(D) = I(G)$. Assume that $U$ is an independent set consisting of sink vertices of $D$ only. Let $\T$ be any non-empty collection of non-empty subsets $W$ of $U$. Then $\Delta (I + \sum_{W \in \T} n(W))$ is acyclic.    
\end{lem}
\begin{proof}We prove by induction on $|U|$. Since $U$ consists of sink vertices of $D$ only, we have $N_D^-(j) = N_D(j) = N_G(j)$ for all $j \in U$. For simplicity of notation, we denote by $N(j)$ the neighborhood of $j$ in this proof. When $U$ consists of a single vertex, say $U = \{u\}$. By assumption, $\T$ is non-empty, hence $\T$ consists of the single set $W = \{u\}$. In this case, we have $n(W) = (w \mid w \in N(u))$. By Lemma \ref{lem_SR}, $\Delta (I + n(W))$ is a cone over $u$. Hence, it is acyclic.

Now, assume that $|U| \ge 2$. By induction, we may assume that every element of $U$ belongs to some subset $W$ in $\T$. For simplicity of notation, assume that $1 \in U$. We define 
$$\U = \{ W \in \T \mid 1 \notin W\} \text{ and } \W = \{ W \subseteq U \setminus \{1\}  \mid \{1 \} \cup W \in \T\}.$$ 
By assumption, $\U$ is a non-empty collection of subsets of $U \setminus \{1\}$. Note that $\W$ might be empty; e.g., when $\{1\} \in \T$. Let $K = I + \sum_{U \in \U} n(U)$. By Lemma \ref{lem_intersection}, we have
$$ I + \sum_{W \in \T} n(W) = \left ( K + n ( \{1\}) \right ) \bigcap \left ( K + \sum_{W \in \W} n(W) \right ).$$

If $\W = \emptyset$ then $I + \sum_{W \in \T} n(W) = K + n ( \{1\})$. In particular, $\Delta(I + \sum_{W \in \T} n(W))$ is a cone over $1$. Hence, we may assume that $\W$ is non-empty. Now, we have $\Delta (K + n ( \{1\}))$ and $\Delta( K + n ( \{1\}) + \sum_{W \in \W} n(W))$ are cones over $1$. By induction, $\Delta(K  + \sum_{W \in \W} n(W))$ is acyclic. By Lemma \ref{lem_SR} and the Mayer-Vietoris sequence, the conclusion follows.
\end{proof}

\begin{proof}[Proof of Theorem \ref{main_thm}] Let $(D,E,\w)$ be a weighted oriented graph. We prove by induction on $|V(D)|$ that $\reg (\overline{I(D,\w)}) \le \reg (I(D,\w))$. The base case where $|V(D)| \le 2$ is clear. Thus, we may assume that $| V(D) | \ge 3$. For simplicity of notation, we set $I = I(D,\w)$ and $J = \overline{ I(D,\w)}$. Let $(\a,i)$ be an extremal exponent of $J$ and $F$ be a face of $\Delta_\a (J)$ such that $\h_{i-1} (\lk_{\Delta_\a(J)} F; \k) \neq 0$. First, we reduce to the case $F = \emptyset$. Assume that $j \in F$. By induction, Lemma \ref{lem_regularity_equality_adding_variable} and Lemma \ref{lem_reg_add_vars}, we have 
$$\reg (J) = \reg (J,x_j) \le \reg (I,x_j) \le \reg (I).$$
Thus, assume that $F = \emptyset$. In other words, 
\begin{equation}\label{eq_non_acyclic}
    \h_{i-1}(\Delta_\a(J);\k) \neq 0.
\end{equation}
Let $U = N^+_D(\supp \a)$. We denote by $D'$ and $\w'$ be the induced subgraph of $D$ and the induced weight function of $\w$ on $V(D) \setminus U$. Let $\b$ be the restriction of $\a$ on $V(D')$. By Lemma \ref{lem_associated_radicals_weight}, Lemma \ref{lem_restriction_weighted}, and Corollary \ref{cor_restriction}, we have 
$$\sqrt{\overline{I(D',\w')} : x^\b} + (x_j \mid j \in U) = \sqrt{\overline{I(D,\w)} : x^\a}.$$
In particular, $\Delta_\a(I(D,\w)) = \Delta_\b(I(D',\w'))$ and $\supp \b$ consists of sink vertices of $D'$ only. By Lemma \ref{lem_associated_radicals_weight}, Lemma \ref{lem_associated_radicals_integral}, Lemma \ref{lem_trivial_homology}, and Eq. \eqref{eq_non_acyclic}, we deduce that $$\sqrt{\overline{I(D',\w')}:x^\b} = I(D') \text{ and } \sqrt{I(D,\w):x^\a} = \sqrt{\overline{I(D,\w)} : x^\a}.$$
Hence, $(\a,i)$ is a critical pair of $I(D,\w)$. By Lemma \ref{Key0}, the conclusion follows.
\end{proof}

\subsection{Edge ideals of weighted oriented complete graphs} In this subsection, we assume that $D$ is an oriented complete graph on $n \ge 2$ vertices. We denote by $|\w| = \sum_{j=1}^n \w(j)$. When $n = 2$, we have $\reg (I(D,\w)) = \reg (\overline{I(D,\w)}) = |\w|$. Thus, we assume that $n \ge 3$ in the sequel. Since $D$ is a complete graph, it can have at most one source vertex. We say that $D$ is of type $1$ if $D$ has a source vertex $u$ and $D\setminus u$ has a source vertex, where $D\setminus u$ is the induced subgraph of $D$ on $V(D) \setminus \{u\}$.

\begin{lem}\label{lem_reg_complete} Let $(D,E,\w)$ be a weighted oriented complete graph on $n \ge 2$ vertices. Then 
$$\reg (I(D,\w)) = \begin{cases}
    |\w| - n +2 & \text { if } D \text{ is of type } 1,\\
    |\w| -n + 1 & \text{ otherwise}.
\end{cases}$$    
\end{lem}
\begin{proof} For simplicity of notation, we denote by $I = I(D,\w)$. Let $(\a,i)$ be an extremal exponent of $I$. Since $\Delta_\a(I)$ is a subcomplex of $\Delta(I)$, which is the disjoint union of $n$ points, we deduce that $i = 0$ or $i = 1$. Furthermore, by Remark \ref{rem_critical}, we have that $a_j \le \w(j) - 1$ for all $j$. Hence, 
$$\reg (I) = |\a| + i + 1 \le |\w| -n + 2.$$
First, assume that $D$ is of type $1$. We may assume that $1$ is a source vertex of $D$ and $2$ is a source vertex of $D\setminus 1$. Let $a_j = \w(j) -1$ for all $j$. Then $1,2\notin \sqrt{I : x^\a}$. In other words, we have $\h_0(\Delta_\a(I);\k) \neq 0$. By Lemma \ref{Key0}, we deduce that $\reg (I) \ge |\w| -n + 2$.

Now, assume that $D$ is not of type 1. There are two cases. 
\medskip

\noindent\textbf{Case 1.} $D$ has no source vertices. Let $a_j = \w(j) -1$ for all $j$. By Lemma \ref{lem_associated_radicals_weight}, $x_j \in \sqrt{I:x^\a}$ for all $j = 1,\ldots, n$. In other words, $\Delta_\a(I)$ is the empty complex. Hence, $(\a,0)$ is a critical pair of $I$. 
\medskip

\noindent\textbf{Case 2.} $D$ has a source vertex. Assume that $1$ is a source vertex of $D$. By assumption, $D\setminus 1$ has no source vertices. Let $a_j = \w(j) -1$ for all $j$. Since $1$ is a source vertex of $D$, $a_1 = 0$. By Lemma \ref{lem_associated_radicals_weight}, $x_j \in \sqrt{I:x^\a}$ for all $j = 2,\ldots, n$. In other words, $\Delta_\a(I) = \{1\}$. Hence, $(\a,0)$ is a critical pair of $I$ with $F = \{1\}$.   
    
In both cases, by Lemma \ref{Key0}, we have $\reg(I) \ge |\w| -n + 1$. Furthermore, for any other exponent $\a \in \Gamma(I)$, we have $|\a| < |\w| -n$. Hence, $\reg (I) \le |\w| -n + 1$. The conclusion follows.
\end{proof}

In \cite{CZW}, Cui, Zhu, and Wei computed the regularity of powers of edge ideals of some weighted oriented complete graphs. Computing exact values of the regularity of all powers of all weighted oriented complete graphs is an interesting problem. We refer to \cite{CZW} for further information. We now prove some preparation lemmas to compute the regularity of integral closures of edge ideals of weighted oriented complete graphs. We say that a vertex $j$ of $D$ is an admissible vertex if it satisfies the following condition. For any vertex $k \in V(D) \setminus \{j\}$, if $k$ is not a source vertex in $D$, then $k$ is not a source vertex in $D \backslash j$. In other words, $j$ is an admissible vertex if $\w(k) = \w'(k)$ for all $k \neq j$, where $\w'$ is the induced weight function on $V(D) \setminus \{j\}$.

\begin{lem}\label{lem_admissible} Let $D$ be an oriented complete graph on $n \ge 4$ vertices. Then $D$ has an admissible vertex.    
\end{lem}
\begin{proof} Assume by contradiction that $D$ has no admissible vertices. Let $X$ be a directed graph on $[n]$ whose $(i,j)$ is a directed edge of $X$ if and only if $(i,j) \in E(D)$ and $N_D^{-}(j) = \{i\}$. Let $i\in [n]$ be an arbitrary vertex. Since $i$ is not admissible, there must exist a vertex $j \in N_D^+(i)$ such that $k \in N_D^+(j)$ for all $k \neq i$. In other words, $(i,j)$ is an edge in $X$. Hence, every vertex in $X$ has an out neighbor. In particular, $X$ has an induced directed cycle $C = v_1,\ldots,v_t$. We claim that $t = 3$. Assume that $t > 3$. Since $D$ is a complete graph, either $(v_1,v_3)$ or $(v_3,v_1) \in E(D)$. If $(v_1,v_3) \in E(D)$, $N_D^-(v_3) = \{v_1,v_2\}$. If $(v_3,v_1) \in E(D)$, $N_D^-(v_1) = \{v_3,v_t\}$, which is a contradiction. Hence, $t = 3$. Since $n \ge 4$, $V(D) \setminus V(C)$ is nonempty. Let $w_1$ be an element of $V(D) \setminus V(C)$. By definition, we must have $w_1 \in N_D^+(v_i)$ for $v_i \in V(C)$. Since $w_1$ is not admissible, there exists $w_2 \notin V(C)$ such that $(w_1,w_2) \in E(X)$. By definition, $N^-_D(w_2) = \{w_1\}$. This is a contradiction, as $v_i \in N_D^{-}(w_2)$ for all $i\in V(C)$. The conclusion follows. 
\end{proof}

\begin{lem}\label{lem_complete_oriented_integral} Let $(D,E,\w)$ be a weighted oriented complete graph on $n \ge 3$ vertices. Assume that $\a \in \NN^n$ is an exponent such that $a_j \le \w(j)$ for all $j = 1, \ldots, n$, $|\a| \ge \max \{ \w(j) \mid j = 1, \ldots, n \} + 1$ and $|\supp \a| \ge 3$. Then $x^\a \in \overline{I(D,\w)}$.     
\end{lem}
\begin{proof} We prove by induction on $n \ge 3$. Let $J = \overline{I(D,\w)}$ and $\omega = \max \{ \w(j) \mid j = 1, \ldots, n\}$. By Corollary \ref{cor_restriction}, we may assume that $\supp \a = [n]$. First, assume that $D$ has a source vertex, say $1$. Then $\w(1) = 1$ and $a_1 = 1$. In other words, we have $\sum_{j = 2}^n a_j \ge \omega$. In particular, we can choose positive rational numbers $c_j$ for $j = 2, \ldots, n$ such that $c_j \le a_j / \w(j)$ and $\sum_{j = 2}^n c_j = 1$. By Lemma \ref{lem_criterion_integral}, we deduce that $x^\a \in J$. Thus, we assume that $D$ has no source vertex. 

We denote by $\e_1,\ldots, \e_n$ the canonical basis of $\RR^n$ and define 
$$Q = \{ \a \in \RR^n \mid 1 \le a_j \le \w(j) \text{ for all } j = 1, \ldots, n \text{ and } |\a| \ge \omega + 1\}.$$
By Lemma \ref{lem_newton_polyhedra} and the fact that $Q$ and $\NP(I(D,\w))$ are convex sets, it suffices to prove that the vertices of $Q$ belong to $\NP(I(D,\w))$. If $a_j = \w(j)$ for some $j$, then $\a \in \NP(I(D,\w))$ as $\supp \a = [n]$ and $D$ has no source vertex. Hence, we may assume that $\w(j) > a_j \ge 1$ for all $j = 1, \ldots, n$.

First, consider the case where $n = 3$. Since $D$ has no source vertex, we may assume that $I(D,\w) = (x_1x_2^{\w(2)},x_2x_3^{\w(3)},x_3x_1^{\w(1)})$ and $\w(3) = \omega$. Let $\a$ be a vertex of $Q$. Since $a_j < \w(j)$ for all $j$, we deduce that $\a = (1,1,\omega-1)$. The other possible vertices $(1,\omega-1,1)$ and $(\omega-1,1,1)$ are of the same form. By Lemma \ref{lem_criterion_integral}, we need to prove that there exist nonnegative numbers $c_1, c_2, c_3$ such that $c_1 + c_2 + c_3 = 1$ and 
$$\a \ge c_1 (\e_1 + \w(2) \e_2) + c_2(\e_2 + \w(3) \e_3) + c_3 (\e_3 + \w(1) \e_1).$$
Let $c_1,c_2,c_3$ be solutions of the following system 
\begin{alignat*}{3}
c_1         &      &+ \w(1) c_3 &= 1 \\
\w(2) c_1   & + c_2 &       & = 1 \\
c_1          & + c_2 & + c_3 & = 1.
\end{alignat*}
Then 
$$c_3 = \frac{\w(2) - 1}{\w(2) \w(1) - \w(1) + 1} \text { and } c_2 = (\w(1) - 1) c_3.$$
We need to prove that $\omega c_2 + c_3 \le \omega - 1$. Equivalently, 
$$(\w(2) - 1) (\omega \w(1) - \omega + 1) \le (\omega - 1) (\w(2) \w(1) - \w(1) + 1).$$
This is equivalent to the condition that $\w(2) \le \w(2)(\omega - \w(1)) + \w(1)$ which is clear as $\omega = \max \{ \w(j) \mid j = 1, 2,3\}$. 

Now, assume that $n \ge 4$ and $\a$ is an exponent in $Q$. By Lemma \ref{lem_admissible}, $D$ has an admissible vertex. Since $D$ has no source vertex, we may assume that $1$ is admissible and $2 \in N^-_D(1)$. Let $D' = D \setminus 1$ be the induced subgraph of $D$ on $V(D) \setminus \{1\}$. Then $D'$ has no source vertices and $\w'(j) = \w(j)$ for all $j = 2, \ldots, n$, where $\w'$ is the induced weight function of $\w$ on $V(D')$. Let 
$$\b = \a - \frac{a_1}{\w(1)} (\w(1) \e_1 + \e_2) = (0,a_2- \frac{a_1}{\w(1)},a_3,\ldots,a_n).$$ 
By Lemma \ref{lem_criterion_integral}, it suffices to prove that there exist nonnegative numbers $\gamma_j$ and exponents $\c_j$ of $E(I(D',\w'))$ such that $\sum_j \gamma_j = 1 - \frac{a_1}{\w(1)}$ and $\b \ge \sum_j \gamma_j \c_j$. Equivalently, $\b' = (\w(1) a_2 - a_1, \w(1) a_3, \ldots, \w(1) a_n) \in \NP(I(D',(\w(1)-a_1)\w'))$, where $(\w(1)-a_1) \w'$ is the weight function on $V(D')$ obtained from $\w'$ by scaling every weight by $\w(1) - a_1$. For simplicity of notation, we set $J' = I(D',(\w(1)- a_1)\w')$. If $b'_j \ge (\w(1) - a_1) \w(j)$ for some $j$, then $\b' \in \NP(J')$ as $b'_j \ge 1$ for all $j = 2, \ldots, n$ and $D'$ has no source vertex. Hence, we may assume that $b'_j < (\w(1) - a_1) \w(j)$. Furthermore, 
$$|\b'| = \w(1) (|\a| - a_1) - a_1 \ge (\w(1) - a_1) \omega + 1.$$
Hence, by induction, $\b' \in \NP(J')$. The conclusion follows.    
\end{proof}

\begin{proof}[Proof of Theorem \ref{thm_complete}] We may assume that $n \ge 3$. Let $J = \overline{I(D,\w)}$ and $\omega = \max \{ \w(j) \mid j = 1, \ldots, n\}$. Let $(\a,i)$ be an extremal exponent of $J$. By degree reason and Lemma \ref{Key0}, it suffices to prove that $|\a| + i \le \omega$. Since $\Delta_\a(J)$ is a subcomplex of $\Delta(I(D))$ which is the disjoint union of $n$ points, we deduce that $i = 0$ or $i = 1$. Hence, we may assume that $|\a| \ge \omega$ and we need to prove that we must have $|\a| = \omega$ and $i = 0$.

By Remark \ref{rem_critical}, $a_j \le \w(j) - 1$ for all $j = 1, \ldots, n$. Hence, $|\supp \a| \ge 2$. There are two cases. 

\medskip
\noindent\textbf{Case 1.} $|\supp \a| \ge 3$. By Lemma \ref{lem_complete_oriented_integral} and the fact that $x^\a \notin J$, we deduce that $|\a| = \omega$. Furthermore, Lemma \ref{lem_complete_oriented_integral} also implies that $x^\a x_j \in J$ for all $j = 1, \ldots, n$. In other words, $\Delta_\a(J)$ is the empty simplicial complex. Hence, $i = 0$. 

\medskip
\noindent\textbf{Case 2.} $|\supp \a| = 2$. We may assume that $\supp \a = \{1,2\}$ and $(2,1) \in E(D)$. By Lemma \ref{lem_complete_oriented_integral} and Lemma \ref{lem_associated_radicals_weight}, we deduce that $x_j \in \sqrt{J:x^\a}$ for all $j \neq 2$. Furthermore, by Corollary \ref{cor_restriction}, we deduce that $x_2 \notin \sqrt{J:x^\a}$. In other words, $\Delta_\a(J) = \{2\}$. Since $2 \in \supp \a$, this is a contradiction, as $\Delta_\a(J)$ is a cone over $2$. The conclusion follows.
\end{proof}

\begin{rem}
\begin{enumerate}
\item The inequality in Theorem \ref{main_thm} is in general strict, even for complete graphs, as shown by Theorem \ref{thm_complete} and Lemma \ref{lem_reg_complete}. For example, consider the ideal $I = (x_1x_2^3,x_2x_3^5,x_3 x_1^6)$. Then $\reg (\overline{I}) = 7$ and $\reg(I) = 12$.
    \item Duan, Zhu, Cui, and Li \cite{DZCL} classified all integrally closed edge ideals of edge-weighted graphs. Most of them are also not integrally closed when the weights are nontrivial. Note that the integral closure of $I(G,\w)$ when $\w(e) = t$ for all edges $e \in E(G)$ is the same as the integral closure of $I(G)^t$. Hence, proving Conjecture \ref{conj_KP} for edge ideals of edge-weighted graphs is an interesting problem.
    \item By the results of \cite{MV1, MV2}, it is expected that the regularity of integral closures of weighted hypergraphs associated with one-dimensional simplicial complexes are nice. We will carry this analysis in subsequent work.
\end{enumerate}
    
\end{rem}
\subsection*{Acknowledgments} We thank Professor K\"uronya for explaining the intuition behind the Conjecture \ref{conj_KP} and for valuable discussions on the preliminary version of the paper. Guangjun Zhu is supported by the Natural Science Foundation of Jiangsu
Province (No. BK20221353).

\end{document}